\documentclass[a4paper,11pt,twoside]{amsart}

\usepackage[latin1]{inputenc}
\usepackage{amsmath, amsfonts, amssymb,amsthm}
\usepackage[mathscr]{eucal} 
\usepackage[pdftex]{graphicx}
\usepackage{color}

\usepackage[T1]{fontenc}
\usepackage[sc]{mathpazo}
\usepackage[all]{xy}
\usepackage{hyperref}
\usepackage[flenq]{nccmath}

\usepackage {anysize}
\marginsize{2cm}{2cm}{2cm}{3cm}

\usepackage{enumerate}

\definecolor{alert}{rgb}{0.8,0,0}

\newcommand{\R}{\mathbb{R}}

\newcommand{\h}{\mathbb{H}}

\newcommand{\M}{\mathbb{M}}

\newcommand{\mf}{\mathbb{M}\times_f\mathbb{R}}

\newcommand{\grad}{\mathrm{grad}}

\renewcommand{\div}{\mathrm{div}}

\newtheorem{theorem}{Theorem}[section]

\newtheorem{corollary}[theorem]{Corollary}
\newtheorem{lemma}[theorem]{Lemma}

\theoremstyle{definition}
  \newtheorem{definition}[theorem]{Definition}

\theoremstyle{remark}

\numberwithin{equation}{section}

\title[Warped Product]{Height estimates for $H$-surfaces in the warped product $\mf$}

\author[]{Abigail Folha}
\address{Universidade Federal Fluminense}
\email{abigailfolha@vm.uff.br}

\author{Carlos Pe\~{n}afiel}
\address{Universidade Federal de Rio de Janeiro}
\email{penafiel@im.ufrj.br}

\author{Walcy Santos}
\address{Universidade Federal de Rio de Janeiro}
\email{walcy@im.ufrj.br}



\begin{document}

\begin{abstract}
In this article, we consider compact surfaces $\Sigma$ having constant mean curvature $H$ ($H$-surfaces) whose boundary $\Gamma=\partial\Sigma\subset \M_0=\M\times_f\{0\}$ is transversal to the slice $\M_0$ of the warped product $\mf$, here $\M$ denotes a Hadamard surface. 
We obtain height estimate for a such surface $\Sigma$ having positive constant mean curvature involving the area of a part of $\Sigma$ above of $\M_0$ and the volume it bounds. Also we give general conditions for the existence of rotationally-invariant topological spheres having positive constant mean curvature $H$ in the warped product $\h\times_f\R$, where $\h$ denotes the hyperbolic disc. Finally we present a non-trivial example of such spheres.
 \end{abstract}

\maketitle

\section{Introduction}
It is a classical result that a compact graph with positive constant mean curvature $H$ in the Euclidean three-dimensional space  $\R^3$ and boundary on a plane can reach at most a height $1/H$ from this plane. Actually, this estimate is optimal because it is attained by the hemisphere of radius $1/H$, this classical result was proved in \cite{H}. More recently, in \cite{AEG}, the authors have obtained height estimates for compact embedded surfaces with positive constant mean curvature in a Riemannian product space $\M\times\R$ (here $\M$ denotes a Riemannian surface without boundary) and boundary on a slice.  In particular, they have obtained optimal height estimate for the homogeneous space $\h\times\R$ (here $\h$ denotes the complete connected simply connected hyperbolic disc having constant curvature $\kappa_g=-1$). The existence of height estimates for surfaces in a 3-dimensional ambient space reveals, in general, important properties on the geometric behaviour of these surfaces as well as existence and uniqueness results.

On the other hand,  in \cite{CR}, the authors have obtained height estimates for positive constant mean curvature, compact  embedded surfaces in the product space $\M^2\times\R$,  whose boundary lies in a slice $\M^2_0$, here $\M^2$ denotes a Hadamard surfaces. They have obtained a relation between the height, the area above this slice and the volume it bounds. They were inspired by \cite{LM}. In this article, we generalize this height estimate for the warped product $\mf$. We obtain the following estimate. 

\begin{theorem}
Let $\M$ be a Hadamard surface whose sectional curvature $K(\M)$ satisfies $K(\M)\le-\kappa\le0$ and let $\Sigma$ be a compact $H$-surface embedded in the warped product $\mf$, with boundary belonging to the slice $\M_0=\M\times_f\{0\}$ and transverse to $\M_0$. If $h$ denotes the height of $\Sigma_1$ with respect to $\M_0$, we have that
\begin{equation*}
h\le \frac{H\mathcal{F}A^{+}}{2\pi}-\kappa\frac{Vol(U_1)}{4\pi}.
\end{equation*}
where $\mathcal{F}$, $A^+$ and $U_1$ are defined in Section \ref{mainsection}. The equality holds if, and only if, $K(\M)\equiv -\kappa$ inside $U_1$ and $\Sigma$ is foliated by circles. Moreover if equality holds, then $f$ is constant on $B$.
\end{theorem}

Also, we focus our attention in the study of existence of rotationally-invariant spheres, which are compact and embedded in the warped product $\h\times_f\R$,  having positive constant mean curvature $H$. We give conditions to the existence of such spheres (see Theorem \ref{t2}) and we present a non-trivial warping function $f$, whose associated warped product $\mf$ admits constant mean curvature, embedded, compact spheres, see Corollary \ref{c1}. 

This article is organized as follows. In Section \ref{preliminares}, we collect some results which are used along this work. In Section \ref{mainsection} we stablish our main result, the height estimate. In Section \ref{meancurvatureequation} we study the  mean curvature equation for surfaces immersed in $\mf$. On the other hand, we give conditions to the existence of rotationally invariant  topological spheres in the warped product $\h\times_f\mathbb{R}$ having positive constant mean curvature. We conclude this section constructing an example of such a topological sphere.

\section{Preliminares}\label{preliminares}

Let $\M$ be a Hadamard surface, that is, a complete, simply connected, two dimensional Riemannian manifold, whose sectional curvature $K(\M)$ satisfies $K(\M)\le-\kappa\le0$, for some constant $\kappa\ge0$. We denote by $\R$ the set of real numbers. On a tri-dimensional Riemannian product $\M\times\R$ we consider the canonical projections $\pi_1:\M\times\R \to \M$ and $\pi_2:\M\times\R \to \R$ defined by $\pi_1(p,t)=p$ and $\pi_2(p,t)=t$ respectively. 

\begin{definition}\label{d1}
Suppose $\M$ is a Hadamard surface endowed with Riemannian metric $g_\M$ and as usual, the real line $\R$ is endowed with the canonical metric $g_0=dt^2$. Let $f:\M \to \R$ be a smooth real function. The warped product $\mf$ is the product manifold $\M\times\R$ endowed with metric
$$g=\pi_1^*(g_\M)+(e^f\circ\pi_1)^2\pi_2^*(g_0).$$
Explicitly, if $v$ is a tangent vector to $\mf$ at $(p,t)$, then
$$g(v,v)=g_\M(d\pi_1(v),d\pi_1(v))+e^{2f(p)}g_0(d\pi_2(v),d\pi_2(v)).$$
\end{definition}
As usual $\M$ is called the base of the warped product $\mf$, $\R$ the fiber and $f$ the warping function.

Note that
\begin{itemize}
\item[(i)] For each $t$, the map $\pi_1\vert_{(\M\times_f\{t\})}$ is an isometry onto $\M$.
\item[(ii)] For each $p$, the map $\pi_2\vert_{(\{p\}\times_f\R)}$ is a positive homothety onto $\R$ with scale factor $e^{-f(p)}$.
\item[(iii)] For each $(p,t)\in\mf$, the slice $\M\times_f\{t\}$ and the fiber $\{p\}\times_f\R$ are orthogonal at $(p,t)$. 
\end{itemize}
 
We denote by $\mathfrak{L}(\M)\subset\chi(\mf)$ and by $\mathfrak{L}(\R)\subset\chi(\mf)$  the set of all lifts of vector fields of $\chi(\M)$ and of $\chi(\R)$ respectively. Let $\xi=\textit{lift}(\partial_t)$, where $\partial_t$ denotes a unit tangent vector field to the real line $\R$. By abuse of notation, we will use the same notation for a vector field and for its lifting. Sometimes we will use the over-bar to emphasize the lift of a vector field. We say that a vector field $X\in\chi(\mf)$ is vertical if $X$ is a non-zero multiple of $\xi$. If $g(X,\xi)=0$, $X$ is said to be horizontal. For a vector field $Z\in\chi(\mf)$, we denote $\vert| Z\vert|:=g(Z,Z)^{1/2}$.  

 We denote by $\nabla$, $\nabla^\R$, $\overline{\nabla}$ the Levi-Civita connection of $\M$, $\R$ and $\mf$ respectively.\\
A straightforward computation gives the following lemma. See \cite{O}.
\begin{lemma}\label{l1}
On $\mf$, if $X,Y\in\mathfrak{L}(\M)$ and $V,W\in\mathfrak{L}(\R)$, then
\begin{enumerate}
\item $\overline{\nabla}_XY\in\mathfrak{L}(\M)$ is the lift of $\nabla_XY$ on $\M$, that is, $\overline{\nabla}_XY=\overline{\nabla_XY}$. 
\item $\overline{\nabla}_XV=\overline{\nabla}_VX=\left(Xf\right)V$.
\item $\overline{\nabla}_WV=-g(V,W)grad(f)+\overline{\nabla^\R_WV}$.
\end{enumerate}
\end{lemma}
As an immediate consequence of Lemma \ref{l1}, the vertical field $\xi$ is a Killing vector field.

In this section we recall some important results which will be used in the proof of the main theorem.  

\subsection{Coarea Formula}
Let $h$ be a proper smooth real function defined on a Riemannian manifold $(M,g)$. Then the set of critical values of $h$ is a null set of $\R$ and the set $O$ of regular values is a open subset of $\R$. For $t\in O$, $h^{-1}(t)$ is a compact hypersurface of $M$, and the gradient vector $\grad(h)(q)$, $h(q)=t$, is perpendicular to $h^{-1}(t)$. Now, we set
\begin{eqnarray*}
\Omega_t & = & \{p\in M;h(p)>t\}, \hspace{.5cm} V_t:=vol(\Omega_t) \\
\Gamma_t & = & \{p\in M;h(p)=t\}, \hspace{.5cm} A_t:=vol_{n-1}(\Gamma_t)
\end{eqnarray*}
See \cite[Theorem 5.8]{S}
\begin{theorem}[Coarea Formula]
The map $t\mapsto V_t$ is of class $C^\infty$ at a regular value $t$ of $h$ such that $V_t<+\infty$, and
$$V^\prime_t=-\int_{\Gamma_t}\|grad(h)\|^{-1}dv_{g_t}$$
where $g_t$ is the induced metric on $\Gamma_t$ from $g$ and $V^\prime(t)=\dfrac{dV}{dt}(t)$.
\end{theorem}

\subsection{The Flux Formula}

Let $U$ be bounded domain in a Riemannian three manifold $M$, whose boundary, $\partial U$, consists of a smooth connected surface $\Sigma$, and the union $Q$ of finitely many smooth,  compact and connected surfaces.   The closed surface $\partial U$ is piecewise-smooth and smooth except perhaps on $\partial\Sigma=\partial Q$. Let
\begin{eqnarray*}
n & = & \textnormal{the outward-pointing unit normal vector field on $\partial U$,} \\
n_\Sigma & = & \textnormal{the restriction of $n$ to $\Sigma$,} \\
n_Q & = & \textnormal{the restriction of $n$ to $Q$,}\\
n_1 & = & \textnormal{the outward-pointing unit conormal to $\Sigma$ along $\partial\Sigma$.}
\end{eqnarray*}
Suppose $Y$ is a vector field defined on a region of $M$ that contains $U$. It was proved in \cite[Propositon 3]{HLR} the following Flux Theorem.

\begin{theorem}[Flux Theorem]
If $Y$ is a Killing vector field on $M$ and $\Sigma$ is a surface having constant mean curvature $H=g(\overrightarrow{H},n)$. Then
$$\displaystyle \int_{\partial\Sigma} Y\cdot n_1+H\int_Q Y\cdot n_Q=0$$
where $Q$, $n_1$ and $\Sigma$ are as defined above.
\end{theorem}

\subsection{Isoperimetric Inequality
for Surfaces}
Let $M$ be a two-dimensional $C^2$-manifold endowed with a $C^2$-Riemannian
metric. We say that $M$ is a generalized surface if the metric in $M$ is allowed to degenerate at isolated points; such points are called singularities of the metric. 

\begin{theorem}[Theorem 1.2,\cite{BM}][Isoperimetric Inequality]
Let $M$ be a generalized surface. Let $D$ be a simply connected domain in $M$ with area $A$ and bounded by a closed piecewise $C^1$-curve $\Gamma$ with length $L$. Let $K$ be the Gaussian curvature and $K_0$ be an arbitrary real number. Assume that in a neighborhood of a singular point, $K$ is bounded above. Then
$$L^2\ge4\pi A\left(1-\displaystyle\frac{1}{2\pi}\int_D(K-K_0)dM -\frac{K_0A}{4\pi}\right),$$
equality holds if, and only if, $K\equiv K_0$ in $D$ and $D$ is a geodesic disc.
\end{theorem}

\section{The main result}\label{mainsection}
Let $\Sigma$ be a compact $H$-surface embedded in $\mf$ with boundary belonging to $\M_0=\M\times_f\{0\}$. Let $\Gamma$ be the boundary of $\Sigma$, $\Gamma=\partial\Sigma$, and assume $\Sigma$ is transverse to $\M_0$ along $\Gamma$.

We denote by $\Sigma^+$ and $\Sigma^-$ the intersection of $\Sigma$ with half-space above and below of $\M_0$ respectively. There is a connected component of $\Sigma^+$ or $\Sigma^-$ that contains $\Gamma$. Without loss of generality, we can assume that $\Gamma\subset\Sigma^+$. We call $\Sigma_1$ the connected component of $\Sigma^+$ that contains $\Gamma$.

Let $\widehat{\Sigma}_1$ be the symmetry of $\Sigma_1$ with respect to $\M_0$. So $\widehat{\Sigma}_1\cup\Sigma_1$ is a compact embedded surface with no boundary, with corners along $\partial\Sigma_1$; this surface bounds a domain $U$ in $\mf$. Let $U_1$ the intersection of $U$ with the half-space above $\M_0$. Thus $U_1$ is a bounded domain in $\mf$, whose boundary $\partial U_1$ consist of the smooth connected surface $\Sigma_1$ and the union $\Omega$ of finitely smooth, compact and connected surfaces in $\M_0$. Let $A^+$ be denote the area of $\Sigma_1$.

Recall, we have considered the projections $\pi_1:\mf \to \M^2$ and $\pi_2:\mf \to \R$ given by $\pi_1(p,t)=p$ and $\pi_2(p,t)=t$ respectively. Let $B=\pi_1(\Sigma_1)$ be the projection on $\M$ of $\Sigma$, since $f$ is smooth, we denote by 
$$\widetilde{\mathcal{F}}=\displaystyle\sup_{p\in B}\left(e^{-2f}\right) \ \ \ \textnormal{and}\ \ \ \mathcal{F}=\widetilde{\mathcal{F}}\cdot \sup_{p\in B}\left(e^f\right).$$
Under these notations, we have the following theorem.

\begin{theorem}\label{t1}
Let $\M$ be a Hadamard surface whose sectional curvature $K(\M)$ satisfies $K(\M)\le-\kappa\le0$ and let $\Sigma$ be a compact $H$-surface embedded in the warped product $\mf$, with boundary belonging to the slice $\M_0=\M\times_f\{0\}$ and transverse to $\M_0$. If $h$ denotes the height of $\Sigma_1$ with respect to $\M_0$, we have that
\begin{equation}\label{e16}
h\le \frac{H\mathcal{F}A^{+}}{2\pi}-\kappa\frac{Vol(U_1)}{4\pi}.
\end{equation}
where $\mathcal{F}$, $A^+$ and $U_1$ are as defined above. The equality holds if, and only if, $K(\M)\equiv -\kappa$ inside $U_1$ and $\Sigma$ is foliated by circles. Moreover if equality holds, then $f$ is constant on $B$.
\end{theorem}
\begin{proof}
We consider the unit normal $N$ of $\Sigma_1$ pointing inside of $U_1$. Let $\overrightarrow{H}$ the mean curvature vector of $\Sigma_1$ and we are supposing that the mean curvature function $H=g(\overrightarrow{H},N)>0$ is a positive constant. We denote by $h:\Sigma^+ \to \R$ the height function of $\Sigma$, that is, $h(p)=\pi_2(p)$ and $h_1=h\vert_ {\Sigma_1}$.

In order to estimate the function $h_1$, let $A(t)$ be the area of $\Sigma_t=\{p\in \Sigma_1;h_1(p)\ge t\}$ and $\Gamma_t=\{p\in \Sigma_1;h_1(p)=t\}$. Then, by the Co-area Formula
$$A^\prime(t)=-\int_{\Gamma_t}\dfrac{1}{\|\grad_f(h_1)\|}dv_{g_t}, \hspace{.2cm} t\in O.$$
where $O$ is the set of all regular values of $\Sigma_1$. And we denote by $L(t)$ the length of the planar curve $\Gamma(t)$, so by the Schwartz inequality 
\begin{equation}\label{e7}
L^2(t)\le \displaystyle\int_{\Gamma(t)}\vert| \grad_f(h_1)\vert| ds_t \int_{\Gamma(t)}\frac{1}{\vert| \grad_f(h_1)\vert|} ds_t=-A^\prime(t)\int_{\Gamma(t)}\vert| \grad_f(h_1)\vert| ds_t, \hspace{.2cm} t\in O.
\end{equation}
On the other hand, we can decompose the vertical Killing field $\xi$ in the tangent and normal projections over the surface $\Sigma$. That is, we can write
\begin{equation}\label{e6}
\xi=T+\nu N
\end{equation}
here $T$ is the tangent projection of $\xi$ and $\nu=g(\xi,N)$ is the normal component of $\xi$ over $\Sigma$. Notice that $\xi=e^{2f}\grad_f(h)$, it follows that 
\begin{equation}\label{e8}
T=e^{2f}\grad_f(h_1)
\end{equation}
which implies $\vert| \grad_f(h_1)\vert|=e^{-2f}\vert| T\vert|$ and by definition of $\widetilde{\mathcal{F}}$ the inequality \eqref{e7} becomes
\begin{equation}\label{e9}
L^2(t)\le -\widetilde{\mathcal{F}}A^\prime(t)\displaystyle\int_{\Gamma(t)} \vert| T\vert| ds_t, \hspace{.2cm} t\in O.
\end{equation}
Furthermore
\begin{equation}\label{e10}
\vert| T\vert|^2=g(\eta^t,\xi)^2
\end{equation}
where $\eta^t$ is the inner conormal of $\Sigma_t$ along $\partial\Sigma_t$. Since $\Sigma_t$ is above the plane $\M_0$ we have $g(\eta^t,\xi)\ge 0$ and therefore $\vert| T\vert|=g(\eta^t,\xi)$. Once here, from \eqref{e9}, we obtain
\begin{equation}\label{e11}
L^2(t)\le -\widetilde{\mathcal{F}}A^\prime(t)\displaystyle\int_{\Gamma(t)} g(\eta^t,\xi) ds_t, \hspace{.2cm} t\in O.
\end{equation}

Let $N_{\Sigma_t}$, $N_{\Omega_t}$ be the unit normal fields to $\Sigma_t$ and $\Omega_t$, respectively, that point inside $U(t)$. Denote by $\eta^t$ the unit conormal to $\Sigma_t$ along $\partial\Sigma_t$, pointing inside $\Sigma_t$. Finally assume that $\Sigma_t$ is a compact surface with constant mean curvature $H=g(\overrightarrow{H},N_{\Sigma_t})>0$. Let $Y$ be a Killing vector field in $\mf$. Then by the Flux Formula
\begin{equation}\label{e12}
\displaystyle \int_{\partial\Sigma_t}g(Y,\eta^t)=2H\int_{\Omega(t)}g(Y,N_{\Omega(t)})
\end{equation}
Taking $Y=\xi$ in \eqref{e12}, we have
$$\displaystyle \int_{\Gamma(t)}g(\xi,\eta^t)\leq 2H\cdot \sup_{p\in B}\left( e^f\right) \cdot \vert|\Omega(t)\vert|$$
where $\vert|\Omega(t)\vert|$is the area of the planar region $\Omega(t)$. Thus if we substitute in \eqref{e11}, we obtain
\begin{equation}\label{e13}
L^2(t)\le-2H\mathcal{F}A^\prime(t)\vert|\Omega(t)\vert|, \hspace{.2cm} \textnormal{for almost every $t\ge0$, $t\in O$.}
\end{equation}
Using the Isoperimetric Inequality for Surfaces, it was proved in \cite{CR}
\begin{equation}\label{e14}
L^2(t)\ge 4\pi\vert|\Omega(t)\vert|+\kappa\vert|\Omega\vert|^2
\end{equation}

From \eqref{e13} and \eqref{e14}, we obtain
$$4\pi\vert|\Omega(t)\vert|+\kappa\vert|\Omega\vert|^2 \le -2H\mathcal{F}A^\prime(t)\vert|\Omega(t)\vert|$$
$$\vert|\Omega(t)\vert|\left(4\pi+\kappa\vert|\Omega\vert| +2H\mathcal{F}A^\prime(t)\right)\le0$$
\begin{equation}\label{e15}
4\pi+\kappa\vert|\Omega\vert| +2H\mathcal{F}A^\prime(t)\le0
\end{equation}
By integrating inequality \eqref{e15} from $0$ to $h=\max_{p\in\Sigma}h_1(p)\ge0$, we obtain
$$4\pi h+2H\mathcal{F}(A(h)-A(0))+\kappa Vol(U_1)\le0$$
therefore
$$A^+=A(0)\ge \frac{2\pi h}{H\mathcal{F}}+\frac{\kappa Vol(U_1)}{2H\mathcal{F}}$$
which is equivalent o inequality \eqref{e16}.

If the equality holds, then all the above inequalities become equalities. In particular, by Isoperimetric Inequality for Surfaces Theorem, $\Gamma(t)$ is the boundary of a geodesic disc in $\M\times_f\{t\}$, for every $t\ge0$, and $K(\M)(p)\equiv-\kappa$ for all $p\in U$.

On the other hand, if equality holds on \eqref{e16}, the  inequality \eqref{e9} is a equality which implies that $e^{-2f}\equiv \widetilde{\mathcal{F}}$ and then the warping function $f$ is constant on $B$.    

\end{proof}

We have the following consequences.

\begin{corollary}\label{c2}
Let $\Sigma$ be a compact, without boundary, embedded surface in the warped product $\mf$, having constant mean curvature $H>0$ and area $A$. Let $U$ be the compact domain bounded by $\Sigma$, then $\Sigma$ lies in a  horizontal slab having height less than 
$\frac{H\mathcal{F}A}{\pi}-\kappa\frac{Vol(U)}{2\pi}$, where $\mathcal{F}$ is defined in the previous theorem. Moreover, one has equality if, and only if, $K(\M)\equiv \kappa$ inside $U$ and $\Sigma$ is foliated by circles.
\end{corollary}

\begin{corollary}\label{c3}
Let $\Sigma$ be a compact, embedded surface in the warped product $\mf$, having constant mean curvature $H>0$ with boundary in the slice $\M_0=\M\times_f\{0\}$ and transverse to $\M_0$.Then  
$$\kappa\frac{Vol(U_1)}{4\pi}\le \frac{H\mathcal{F}A^{+}}{2\pi}$$ 
where $\mathcal{F}$, $A^+$ and $U_1$ are defined in the previous theorem. 
\end{corollary}

\section{Mean curvature equation}\label{meancurvatureequation}
Let $\Omega\subset \M$ be a domain and $u:\Omega \to \R$ be a smooth function, The graph of $u$ in $\mf$ is the set
\begin{equation}\label{g1}
\Sigma_u=\{(p,u(p))\in\mf;p\in\Omega\}
\end{equation}
Let $\overrightarrow{H}$ denote the mean curvature vector field of $\Sigma_u$ and we choose a unit normal vector field $\overrightarrow{N}$ to $\Sigma
_u$ satisfying $g(\overrightarrow{N},\xi)\le0$. Throughout this article a surface having constant mean curvature $H$ will be called an $H$-surface. In order to obtain the mean curvature equation in the divergence form, we prove the next lemma.
\begin{lemma}\label{ldiva}Let $X$ be a vector field in $\M\times_f \R$
\begin{equation}\label{diva}
e^{f}\div_f(X)=\div_{\M}(e^{f}d\pi_1(X))+\xi\left(e^{-f}\ g(X,\xi)\right),
\end{equation}
where $\div_f$ and $\div_{\M}$ are the divergence on $\mf$ and $\M$, respectively.
\end{lemma}

\begin{proof} Let $\{x_1,x_2\}$ be local coordinates for $\M$, and $\{x_3=t\}$ a local coordinate for $\R$ whose associated vectors fields are $\{\partial_{x_1}, \partial_{x_2},\partial_{x_3}\}$. Their lifts to $\mf$ are denoted by $\{\overline{\partial}_{x_1}, \overline{\partial}_{x_2}, \overline{\partial}_{x_3}=\xi\}$, these are the associated vector field to the local coordinates $\{x_1,x_2,x_3\}$ of $\mf$. Denoting by $g^{ij}$ the coefficients of the inverse matrix of $g$ and aplying the definition of  the divergence of a vector field $X$ on $\mf$ we obtain

\begin{eqnarray*}
\div_f(X)&=&\dfrac{1}{\sqrt{\mathrm{det} g}}\sum_{i,j=1,2,3} \left(\overline{\partial}_{x_i}\left(\sqrt{\mathrm{det g}}\ g^{ij}\ g(X,\overline{\partial}_{x_j}) \right)\right)\\
&=& \dfrac{1}{e^{f} \ \sqrt{\mathrm{det}g_{\M}}}\left(\sum_{i,j=1,2}\partial_{x_i}\left(e^f\sqrt{\mathrm{det g_{\M}}}\ g_{\M}^{ij}\ g_{\M}(d\pi_1(X),\partial_{x_j}) \right)+\overline{\partial}_{x_3}\left(e^{-f}\sqrt{\mathrm{det g_\M}} \ g(X,\overline{\partial}_{x_3}) \right)\right)\\
&= &\dfrac{1}{\ \sqrt{\mathrm{det}g_{\M}}}\left(\sum_{i,j=1,2}\partial_{x_i}\left( \sqrt{\mathrm{det g_{\M}}}\ g_{\M}^{ij}\ g_{\M}(d\pi_1(X),\partial_{x_j}) \right)+\overline{\partial}_{x_3}\left( e^{-2f} \sqrt{\mathrm{det g_\M}} \ g(X,\overline{\partial}_{x_3}) \right)\right) +\\
& &+\dfrac{1}{e^f\ \sqrt{\mathrm{det}g_{\M}}} \left(\sqrt{\mathrm{det}g_{\M}}\sum_{i,j=1,2} \partial_{x_i}(e^f)\ g_{\M}(d\pi_1(X),\partial_{x_j}) \ g_{\M}^{ij}\ \right)\\
&=&\div_{\M}(d\pi_1(X))+\dfrac{1}{e^f}\ g_{\M}\left( d\pi_1(X),\mathrm{grad}_{\M}(e^f)\right)+\overline{\partial}_{x_3}\left( e^{-2f}\ g(X,\overline{\partial}_{x_3}) \right).
\end{eqnarray*}
The last equality follows from de definition of the gradient on $\M$ which is given by $\mathrm{grad}_{\M}(e^f):=\sum_{i=1,2}\partial_{x_i}(e^f)g_{\M}^{ij}\partial_{x_j}$. Then, 

\begin{equation*}
e^{f}\div_f(X)=\div_{\M}(e^{f}d\pi_1(X))+\xi\left( e^{-f}\ g(X,\xi)\right).
\end{equation*}

\end{proof}
Taking the equation \eqref{diva} into account we obtain the following mean curvature equation for vertical graphs in $\M\times_f\R$

\begin{lemma}\label{l3a}
Let $\Sigma_u\subset\mf$ be the vertical graph of a smooth function $u:\Omega\subset\M \to \R$ having mean curvature function $H$. Then, $u$ satisfies
\begin{equation}\label{e2a}
-2He^f=\div_\M\left(e^f\ \dfrac{\mathrm{grad}_{\M} u}{W}\right), 
\end{equation} 
where $W^2=e^{-2f}+\|\mathrm{grad}_{\M} u\|^2$
\end{lemma}

 \begin{proof}
 We consider the smooth function $u^{*}:\mf\longrightarrow\R$ defined by $u^{*}(x, y, t) = u(x, y)$. Set $F(x, y, t) =u^{*}(x, y, t)-t$, therefore zero is a regular value of $F$ and $F^{-1}(0)=\Sigma_u$. It is well-known that the function $H$ satisfies
\begin{equation}\label{e4a}
    2H=-\div_f\left(\dfrac{\grad_f(F)}{\vert|\grad_f(F)\vert|}\right),
\end{equation}
where $\div_f$ and $\grad_f$ denote the divergence and gradient in $\mf$, respectively. Let $X$ be a vector field on $\M$, we denote its lift to $\mf$ by $\overline{X}$. We have, 
$$\mathrm{grad}_f(F)=\overline{\mathrm{grad}_{\M}(u)} -e^{-2f}\,\xi.$$
Setting $W^2=\vert|\grad_f(F)\vert|^2=e^{-2f}+\|\mathrm{grad}_{\M} u\|^2$ and applying Lemma \ref{ldiva} in equation \eqref{e4a}, we obtain

\begin{eqnarray*}
  2H \ e^{f}&=&-e^{f}\div_f\left(\dfrac{\grad_f(F)}{\vert|\grad_f(F)\vert|}\right)\\
  &=&-\mathrm{div}_{\M}\left(e^f \dfrac{\mathrm{grad}_{\M}(u)}{W}\right)+\xi\left(e^{-f}\ g\left(\dfrac{e^{-2f} \ \xi}{W}, \xi\right)\right)\\
  &=&-\mathrm{div}_{\M}\left(e^f \dfrac{\mathrm{grad}_{\M}(u)}{W}\right).
\end{eqnarray*}

\end{proof}

\subsection{Some $H$-surfaces in $\h\times_f\R$}
Now let us focus on the case $\mathbb{M}=\h$, where $\h$ is the connected, simply connected two-dimensional Hyperbolic disc $\h=\{(x,y)\in\R^2; x^2+y^2<1\}$ having constant curvature $\kappa=-1$. We consider coordinates $(\rho, \theta)$ in $\h$, where $\rho$ is the hyperbolic distance to the origin and $\theta$ is the angle between a segment from the origin and the positive semi-axis $x$. More precisely, we consider a parametrization $\varphi(\rho,\theta)=(\tanh(\rho/2)\cos\theta,\tanh(\rho/2)\sin\theta)$ from $(0,+\infty)\times[0,2\pi)$ in the hyperbolic disc $\h$. For simplicity, we treat properties of the surfaces $\Sigma_u$ using the disc model $\h$ or the slab $(0,+\infty)\times[0,2\pi)$ (via the parametrization $\varphi$). In these polar coordinates the metric on $\h$ is given by

\begin{equation}\label{e3a}
g_{\h}=d\rho^2+\sinh^2(\rho)\ d\theta^2
\end{equation}   
Let $u:\Omega\subset \mathbb{H}\longmapsto \mathbb{R}$ be a smooth function. A parametrization of the graph $\Sigma_u$ given in \eqref{g1} is
$$\psi(\rho,\theta) = (\tanh(\rho/2)\cos\theta,\tanh(\rho/2)\sin\theta,u(\rho,\theta)),$$
where $(\rho,\theta)\in\varphi^{-1}(\Omega)$. Under this notation, we have the next lemma.
\begin{lemma}\label{intpa}
Let $\Sigma_u$ be a surface in $\h\times_f\R$ which is the graph of a smooth function $u:\Omega \to \R$ defined on a domain $\Omega\subset\h$. Suppose the unit normal vector field $\overrightarrow{N}$ of $\Sigma_u$ points downwards in $\h\times_f\R$, that is, $g(\overrightarrow{N},\xi)\le0$ and let $H$ be the mean curvature function. Then $u=u(\rho,\theta)$ satisfies
\begin{equation}
-2H\  e^f=\dfrac{1}{\sinh\rho}\dfrac{\partial}{\partial_{\rho}}\left[ g_\h\left(e^f\ \dfrac{\mathrm{grad}_{\h} u}{W},\partial_\rho\right)\sinh\rho \right] + \dfrac{1}{\sinh^2\rho}\dfrac{\partial}{\partial_{\theta}}\left[ g_\h\left(e^f\ \dfrac{\mathrm{grad}_{\h} u}{W},\partial_\theta\right) \right]
\end{equation}
where $f=f(\rho,\theta)$, $W^2=e^{-2f}+\|\mathrm{grad}_{\h} u\|^2$, and $\dfrac{\partial}{\partial_z}$ denotes the derivative with respect to $z$.
\end{lemma} 
\begin{proof}
From Lemma \ref{e3a} and  Divergence's Theorem
\begin{equation}\label{e17}
-\int_\Omega 2H\, e^fdA=\int_\Omega\mathrm{div}_\h\left(e^f\ \dfrac{\mathrm{grad}_{\h} u}{W}\right)dA=\int_{\partial\Omega}g_\h\left(e^f\ \dfrac{\mathrm{grad}_{\h} u}{W},\eta\right)ds
\end{equation}
where $\Omega$ is a domain with boundary $\partial\Omega$ and $\eta$ is the unit outer-conormal to $\Omega$.

Let us consider the domain $\Omega=[\rho_0,\rho_1]\times[\theta_0,\theta_1]$ in the $\rho\theta$-plane. Notice $\partial\Omega=\gamma_1\cup\gamma_2\cup\gamma_3\cup\gamma_4$, where 
\begin{itemize}
\item $\gamma_1(s)=(s,\theta_0), \hspace{.2cm} \rho_0\le \rho_1$.
\item $\gamma_2(s)=(\rho_1,s), \hspace{.2cm} \theta_0\le \theta_1$.
\item $\gamma_3(s)=(\rho_1-s,\theta_1), \hspace{.2cm} 0\le s\le \rho_1- \rho_0$.
\item $\gamma_4(s)=(\rho_0,\theta_1-s), \hspace{.2cm} 0\le s\le \theta_1-\theta_0$.
\end{itemize}

Therefore
\begin{equation}\label{e21}
-\int_\Omega 2H\, e^fdA=\sum_{i=1}^4\int_{\partial\Omega}g_\h\left(e^f\ \dfrac{\mathrm{grad}_{\h} u}{W},\eta_i\right)ds
\end{equation}
A straightforward computation give us
\begin{enumerate}
\item $\displaystyle{\int_{\gamma_1}}g_\h\left(e^f\ \dfrac{\mathrm{grad}_{\h} u}{W},\eta_1\right)ds=- \displaystyle{\int_{\rho_0}^{\rho_1}}g_\h\left(e^f\ \dfrac{\mathrm{grad}_{\h} u}{W},\partial_\theta\right)(\rho,\theta_0)\dfrac{1}{\sinh\rho}d\rho $
\item $\displaystyle{\int_{\gamma_2}}g_\h\left(e^f\ \dfrac{\mathrm{grad}_{\h} u}{W},\eta_2\right)ds= \displaystyle{\int_{\theta_0}^{\theta_1}}g_\h\left(e^f\ \dfrac{\mathrm{grad}_{\h} u}{W},\partial_\rho\right)(\rho_1,\theta)\sinh\rho_1 d\theta $
\item $\displaystyle{\int_{\gamma_3}}g_\h\left(e^f\ \dfrac{\mathrm{grad}_{\h} u}{W},\eta_3\right)ds= \displaystyle{\int_{\rho_0}^{\rho_1}}g_\h\left(e^f\ \dfrac{\mathrm{grad}_{\h} u}{W},\partial_\theta\right)(\rho,\theta_1)\dfrac{1}{\sinh\rho}d\rho $
\item $\displaystyle{\int_{\gamma_4}}g_\h\left(e^f\ \dfrac{\mathrm{grad}_{\h} u}{W},\eta_4\right)ds=- \displaystyle{\int_{\theta_0}^{\theta_1}}g_\h\left(e^f\ \dfrac{\mathrm{grad}_{\h} u}{W},\partial_\rho\right)(\rho_0,\theta)\sinh\rho_0 d\theta $
\end{enumerate}
Using the last four expressions into equation \eqref{e21}, we obtain
\begin{equation}\label{e22}
-\int_\Omega2H  e^fdA=\int_\Omega\left[\dfrac{1}{\sinh\rho}\dfrac{\partial}{\partial_{\rho}}\left[ g_\h\left(e^f\ \dfrac{\mathrm{grad}_{\h} u}{W},\partial_\rho\right)\sinh\rho \right] + \dfrac{1}{\sinh^2\rho}\dfrac{\partial}{\partial_{\theta}}\left[ g_\h\left(e^f\ \dfrac{\mathrm{grad}_{\h} u}{W},\partial_\theta\right) \right]\right]dA
\end{equation}
Once equation \eqref{e22} holds for any such $\Omega$ we complete the proof.

\end{proof}

\subsection{Rotational spheres in $\h\times_f\R$ }

In this section we will construct rotationally-invariant spheres. In order to do that, the warping function $f$ must depend only on $\rho$ as well as the function $u$, that is, $f(\rho,\theta)\equiv f(\rho)$ and $u(\rho,\theta)\equiv u(\rho)$.

Let $I$ be an interval in the positive $x$-axis and $u: I\to \R$ be a smooth function whose graph $\Gamma_u$ lies in the $xt$-plane. Denote by $\mathfrak{G}$ the group of rotations around the origin of the hyperbolic disc $\h$. The rotationally-invariant surface $\Sigma_u$ obtained from $\Gamma_u$ is the surface $\Sigma_u=\mathfrak{G}\Gamma_u$. We assume that $\Sigma_u$ has non-negative constant mean curvature $H$ with respect to the pointing downwards unit normal vector field $\overrightarrow{N}$ of $\Sigma_u$. In order to give the conditions to the existence of spheres in the warped product $\h\times_f\R$, we need the next definition.

\begin{definition}\label{d2}
Consider the warped product $\h\times_f\R$, with $f\equiv f(\rho)$. Let $H>0$ be a positive constant and $d$ a constant depending on $H$ and $f$. We say the function $u=u(\rho)$ is an admissible solution if the following conditions holds 
\begin{enumerate}
\item There exists an interval $[0,\rho_0], \ \rho_0<\infty,$ such that the function 
$$G(\rho)=d-2HF(\rho), \hspace{.4cm} \textnormal{with} \hspace{.4cm} F_\rho(\rho)=e^{f(\rho)}\sinh\rho$$ 
satisfies $G(0)=0$, $G(\rho)>0$ on $(0,\rho_0]$ and $G(\rho)<e^{f(\rho)}\sinh\rho$, on $(0,\rho_0]$. Where $F_\rho$ denotes the derivative of the function $F$ with respect to $\rho$.
\item The graph $\Gamma_u$ of the function $u$  is  defined on $[0,\rho_0],\  \rho_0<\infty$ and is given by the ordinary differential equation (ODE)
\begin{equation}\label{ea20}
u_\rho(\rho)=\dfrac{G(\rho)}{e^{f(\rho)}\sqrt{e^{2f(\rho)}\sinh^2\rho-G^2(\rho)}}
\end{equation}
which has zero derivative at $\rho=0$, non-finite derivative at $q=(\rho_0,u(\rho_0))$ and finite geodesic curvature at the points $p$ and $q$ as long as the left side of equation \eqref{ea20} is well defined on $\rho=0$.
\end{enumerate}
\end{definition}

Now we presented the main theorem of this section, the notation on Definition \ref{d2} will be used.
\begin{theorem}[Rotational spheres]\label{t2}
Let $f=f(\rho)$ be a warping function depending only on $\rho$. Suppose that for a positive constant $H>0$ there exists a constant $d$ (depending on $H$ and $f$) and an admissible solution $u(\rho)$. Then there exists a rotational sphere $S^H$ having constant mean curvature $H$, which is invariant by the group $\mathfrak{G}$ and up to vertical translations, is a bi-graph with respect to the slice $\h_0=\h\times_f\{0\}$. 
\end{theorem}
\begin{proof}
We have denoted by  $\Sigma_u=\mathfrak{G}\Gamma_u$, the rotationally-invariant surface which is a graph of a function $u(\rho,\theta)\equiv u(\rho)$. If $\Sigma_u$ has non-negative constant mean curvature $H$ with respect to the downwards pointing unit normal vector field $\overrightarrow{N}$ then, by Lemma \ref{intpa} the function $u$ satisfies
\begin{equation}\label{e23}
-2H\  e^f\sinh\rho=\dfrac{\partial}{\partial_{\rho}}\left(\dfrac{e^f\ u_\rho}{W}\sinh\rho \right)
\end{equation}
where $W^2=e^{-2f}+u^2_\rho$. Recall that $F$, by Definition \ref{d2}, satisfies the ODE $F_\rho(\rho)=e^{f(\rho)} \sinh\rho$. Intregating equation \eqref{e23}, we obtain
\begin{equation}\label{e24}
d-2H\  e^fF=\dfrac{e^{2f}\ u_\rho}{\sqrt{1+(e^{f}\ u_\rho)^2}}\sinh\rho
\end{equation}
where $d\in\R$ is a constant. Once we are assuming that there is an admissible solution,  equation \eqref{e24} is equivalent to
\begin{equation}\label{e25}
u_\rho(\rho)=\pm\dfrac{d-2HF}{e^{f}\sqrt{e^{2f}\sinh^2\rho-(d-2HF)^2}}
\end{equation}
The admissible solution $u$  is a solution for equation \eqref{e25}. We can glue together the graph of the $u$ with the graph of $-u$ in order to obtain the rotational sphere $S^H$.   
\end{proof}

\subsection{Examples of rotational spheres}
Once rotationally-invariant spheres do not have to exists for every warping function $f=f(\rho)$. To see that the set of admissible solutions is non-empty, we consider the non-trivial warping function $f(\rho)$ given by
$$f(\rho)=\ln\left(2\cosh(\rho)\right)$$
From Lemma \ref{intpa}, the function $u$ which generates the rotationally-invariant surface $\Sigma_u=\mathfrak{G}\Gamma_u$ having constant mean curvatura $H$,  satisfies
\begin{equation}\label{e18}
-2H\sinh(2\rho) \ = \ \partial_{\rho}\left( \dfrac{e^{2f} \, u_\rho}{\sqrt{1+(e^{f}u_{\rho})^2}}\sinh\rho \  \right)
\end{equation}
By integrating equation \eqref{e18}, we obtain
\begin{equation}\label{FR}
    u_\rho(\rho)=\pm\frac{d-H\cosh(2\rho)}{2\cosh(\rho)\sqrt{\sinh^{2}(2\rho)-(d-H\cosh(2\rho))^{2}}}
\end{equation}
where $d\in\R$. Once here, we have the next corollary.

\begin{corollary}\label{c1}
Consider the warping function $f(\rho)=\ln\left(2\cosh(\rho)\right)$, where $\rho\ge0$ is the hyperbolic distance from the origin in the hyperbolic disk $\h$. For each positive constant $H>1$, there exists a rotational sphere $S^H$, embedded in the warped product $\h\times_f\R$ having constant mean curvature $H$,  which is invariant by the group $\mathfrak{G}$ and up to vertical translation, the sphere $S^H$ is a bi-graph with respect to the slice $\h_0=\h\times_f\{0\}$.

\end{corollary}
\begin{proof}
For the product space $\h\times\R$, it was proved in \cite{ST} (or \cite{CP2}) the existence of an admissible solution which generates an embedded, compact, rotationally-invariant sphere $S$ having constant mean curvature $H$, for any constant satisfying $2H>1$. If we denote by $u_0$ this admissible solution, then $u_0$ satisfies the ODE
\begin{equation}\label{eis}
    (u_0)_\rho(\rho)=\frac{(2H-2H\cosh(\rho))}{\sqrt{\sinh^{2}(\rho)-(2H-2H\cosh(\rho))^{2}}},
\end{equation}
here $\rho\in[0,\rho_0]$, for some fixed $\rho_0>0$.

Take $d=H$, for $H>1$ in equation \eqref{FR}, following the same ideas presented in \cite{ST}, we see that the solution of equation \eqref{FR} is an admissible solution over an interval $[0,\rho_1]$ for some fixed $\rho_1>0$. Notice that, there exists constants $m$ and $M$ such that
\begin{equation}\label{FR1}
    \frac{H-H\cosh(2\rho)}{m\sqrt{\sinh^{2}(2\rho)-(H-H\cosh(2\rho))^{2}}}\le u \le  \frac{H-H\cosh(2\rho)}{M\sqrt{\sinh^{2}(2\rho)-(d-H\cosh(2\rho))^{2}}}
\end{equation}
in $[0,\rho_1]$. By Theorem \ref{t2}, the admissible solution $u$ generates the sphere $S^H$.

\end{proof}



\end{document}